\documentclass{amsart}
\usepackage{amsmath,amsthm,amssymb}
\usepackage[latin1]{inputenc}
\usepackage{graphicx}
\usepackage[T1]{fontenc}
\usepackage{enumerate}
\usepackage{float}
\usepackage{curves}
\usepackage{rotfloat}

\newtheorem{theorem}{Theorem}[section]
\newtheorem{corollary}[theorem]{Corollary}
\newtheorem{lemma}[theorem]{Lemma}
\newtheorem{proposition}[theorem]{Proposition}

\theoremstyle{definition}
\newtheorem{definition}[theorem]{Definition}
\newtheorem{remark}[theorem]{Remark}

\newcommand\Rt{\operatorname{\mathbf R}}	% Transitive reduction
\newcommand\Ft{\operatorname {\mathbf{Tr}}}	% Transitive closure
\newcommand\ft{\operatorname{\mathbf{tr}}}	% Relative transitive closure
\newcommand\RE{\operatorname{RE}}		% Redundant edge set

\numberwithin{equation}{section}

\begin{document}

\title{Transitive Closure and Transitive Reduction in Bidirected Graphs}

\author{Ouahiba  Bessouf, Alger}
\author{Abdelkader  Khelladi, Alger} 
\author{Thomas  Zaslavsky, Binghamton}

\begin{abstract}
   In a bidirected graph an edge has a direction at each end, so bidirected graphs generalize directed graphs. We generalize the definitions of transitive closure and transitive reduction from directed graphs to bidirected graphs by introducing new notions of bipath and bicircuit that generalize directed paths and cycles. We show how transitive reduction is related to transitive closure and to the matroids of the signed graph corresponding to the bidirected graph.
\end{abstract}



\maketitle

\pagestyle{myheadings}
\markright{\textsc{Transitive Closure and Transitive Reduction in Bidirected Graphs}}\markleft{\textsc{Bessouf, Khelladi, and Zaslavsky}}

%%%%%%%%%%%%%%%%%%%%%%%%%%%%
%%%%%%%%%%%%%%%%%%%%%%%%%%%%

\section{Introduction.}

Bidirected graphs are a generalization of undirected and directed graphs. Harary defined
in 1954 the notion of signed graph. For any bidirected graph,
we can associate a signed graph of which the bidirected graph is an orientation. Reciprocally, any signed graph
can be associated to a bidirected graph in multiple ways, just as a graph can be associated to a directed graph. 

Transitive closure is well known.  Transitive reduction in directed graphs was introduced by Aho, Garey, and Ullman \cite{agu}. 
The aim of this paper is to extend the concepts of transitive closure, 
which is denoted by $\Ft(G_{\tau})$, and transitive reduction, which is denoted by $\Rt(G_{\tau})$, to
bidirected graphs. We seek to find definitions of transitive closure and transitive reduction 
for bidirected graphs through which the classical concepts would be a special case. 
We establish for bidirected graphs some properties of these concepts and a 
duality relationship between transitive closure and transitive reduction.

%222222222222222222222222222222222222222222222222222222222222222222222222222222222222222222222222222222222222222222222222222222222222222222
\section{Bidirected Graphs}

We allow graphs to have loops and multiple edges.  Given an undirected graph $G = (V,E)$, the set of \emph{half-edges}
of $G$ is the set {$\Phi(G)$} defined as follows:
$$\Phi(G)= \{(e, x)\in E \times V : e \text{ is incident with } x \}.$$
Thus, each edge $e$ with ends $x$ and $y$ is represented by its two half-edges $(e, x)$ and $(e, y)$.  
For a loop the notation does not distinguish between its two half-edges.   There is no very good notation for the two half-edges of a loop, but we believe the reader will be able to interpret our formulas for loops.

A \emph{chain} (or \emph{walk}) is a sequence of vertices and edges, $x_0, e_1, x_1, \ldots, e_k, x_k$, such that $k\geq0$ and $x_{i-1}$ and $x_i$ are the ends of $e_i$ $\forall\ i=1,\ldots,k$.  It is \emph{elementary} (or a \emph{path}) if it does not repeat any vertices or edges.  It is \emph{closed} if $x_0=x_k$ and $k>0$.  A \emph{partial graph} of a graph is also known as a \emph{spanning subgraph}, i.e., it is a subgraph that contains all vertices.  The terminology is due to Berge \cite{berge}.

\subsection{Basic Properties of Bidirected Graphs}

\begin{definition}\label{1}
{
A \emph{biorientation} of $G$ is a signature of its half-edges:
$${\tau :\Phi(G)\rightarrow\{- 1, +1\}}.$$
It is agreed that $\tau(e, x) = 0$ if $(e, x)$ is not a half-edge of $G$; that makes it possible to extend $\tau$ to
all of $E\times V$, which we will do henceforth.

A \emph{bidirected graph} is a graph provided with a biorientation; it is written $G_{\tau} = (V, E; \tau)$.
}
\end{definition}

\begin{definition}\label{1a}
{
An edge $e=\{x,y\}$ in a bidirected graph is notated $e=\{x^\alpha,y^\beta\}$ if $\tau(e,x)=\alpha$ and $\tau(e,y)=\beta$.  Two edges $e,f$ both notated $\{x^\alpha,y^\beta\}$ are called \emph{parallel}.
}
\end{definition}

\setlength{\unitlength}{.8mm}

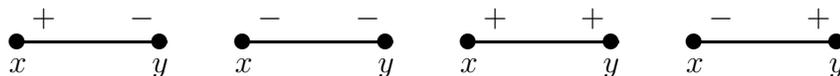
\begin{figure}[h]
\begin{center}
\begin{picture}(110,10)(0,7)
\put(0,10){\circle*{2}}
\put(19,10){\circle*{2}}
\put(0,10){\line(1,0){20}}
\put(2,12){$+$}
\put(15,12){$-$}
\put(-1,6){$x$}
\put(18,6){$y$}
\put(30,10){\circle*{2}}
\put(49,10){\circle*{2}}
\put(30,10){\line(1,0){20}}
\put(32,12){$-$}
\put(45,12){$-$}
\put(29,6){$x$}
\put(48,6){$y$}
\put(60,10){\circle*{2}}
\put(79,10){\circle*{2}}
\put(60,10){\line(1,0){20}}
\put(62,12){$+$}
\put(75,12){$+$}
\put(59,6){$x$}
\put(78,6){$y$}
\put(90,10){\circle*{2}}
\put(109,10){\circle*{2}}
\put(90,10){\line(1,0){20}}
\put(92,12){$-$}
\put(105,12){$+$}
\put(89,6){$x$}
\put(108,6){$y$}
\end{picture}
\end{center}
\caption{The four possible biorientations of an edge $\{x, y\}$ of $G_{\tau}$.}\label{FIG1}
\end{figure}

Each edge (including a loop) has four possible biorientations (figure \ref{FIG1}); therefore, the number of biorientations of $G$ is $4^{|E|}$.  

\begin{definition}\label{1ssi}
{\rm
We define two subsets of $V$:
\begin{align*}
V_{+1} &= \{x \in G_{\tau} : \tau (e, x) = +1, \ \forall (e, x) \in \Phi_{x}\} \text{ is the set of \emph{source} vertices},\\
V_{-1} &= \{x \in G_{\tau} : \tau (e, x) = -1, \ \forall (e, x) \in \Phi_{x}\} \text{ is the set of \emph{sink} vertices},
\end{align*}
where $\Phi_{x}$ is the set of all half-edges incident with $x$.  (Note that this is the opposite convention for arrows to that in \cite{zo}.)
}
\end{definition}

We observe that $V_{+1}\cap V_{-1}$ is the set of vertices that are not an end of any edge (\emph{isolated} vertices).

\begin{definition}\label{2} {\rm\cite{b}}
{\rm
Let $G_{\tau} = (V, E; \tau)$ be a bidirected graph. Then $W$ (resp.,
$\overline{W}$) is a function defined on $V$ (resp., $E$) as follows:
\begin{center}
$W  : V \rightarrow \mathbb Z,$
\hspace{1cm} $ x \mapsto W(x) =\sum_{e\in E} \tau(e,x), $
\end{center}
\begin{center}
$\overline{W} : E  \rightarrow   \{- 2, 0, 2\},$ \hspace{1cm} $ e
\mapsto \overline {W}(e) = \sum_{x\in V}\tau(e, x).$
\end{center}
Thus, for a vertex $x$, $W(x)$ is the number of positive half-edges incident with $x$, minus the number of negative half-edges incident with $x$.
}
\end{definition}

\begin{definition}{\rm\cite{hn}}\label{3}
{\rm
A \emph{signed graph} is a triple $(V, E; \sigma)$ where $G = (V, E)$
is an undirected graph and $\sigma$ is  a signature of the edge set $E$:
\begin{center}
$\sigma : E  \rightarrow \{ -1, +1\}$.
\end{center}
 A signed graph is written $G_{\sigma} = (V, E; \sigma)$.
 }
\end{definition}

\begin{definition}\label{10}{\rm\cite{b}}
{\rm
Let $G_{\sigma}=(V, E; \sigma )$ be a signed graph and $P$ a chain (not necessarily elementary)
connecting $x$ and $y$  in $G_{\sigma}$:
$$P:x,e_{1},x_{1},e_{2},x_{2},\ldots,y,$$
where $x,x_{1},\ldots,y$ are vertices  and $e_{1},e_{2},\ldots$ are  edges of $G$.
We put $$\sigma(P)=\displaystyle{\prod_{e_{i}\in P}\sigma(e_{i})}.$$
We write $P^{\alpha}$ instead of $P$, if $\alpha=\sigma(P)$.
$P^{\alpha}$ is called a \emph{signed chain} of sign $\alpha$ connecting $x$ and $y$.
A signed chain is \emph{minimal} if it contains no signed chain with the same ends and the same sign.
See figure \ref{FIG3}.

\begin{figure}[htbp]
  \begin{center}
  \includegraphics[scale=1]{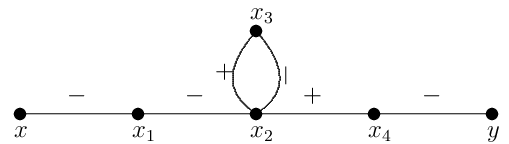}\\
$P^{+}:x,x_{1},x_{2},x_{3},x_{2},x_{4},y$ is a positive signed chain.\\
$P^{-}:x,x_{1},x_{2},x_{4},y$ is a negative signed chain.\\
  \end{center}
\caption{$P^{+}$ contains $P^{-}$ as a subchain, but both $P^{+}$ and $P^{-}$ are minimal signed chains from $x$ to $y$ because their signs differ.}\label{FIG3}
\end{figure}

Especially, a cycle in a signed graph is \emph{positive} if the number of its negative edges is even.  In the opposite case, it is \emph{negative}.
}
\end{definition}

\begin{definition}\label{5}
{\rm
A signed graph is \emph{balanced} if all its cycles are positive \cite{hn}.  

A signed graph is \emph{antibalanced} if, by negating the signs of all edges, it becomes balanced \cite{hd}.
}
\end{definition}

It follows from the definitions that a cycle is balanced if, and only if, it is positive.

\begin{lemma}\label{anti}
A signed graph is antibalanced if, and only if, every positive cycle has even length and every negative cycle has odd length.
\end{lemma}
\begin{proof}
Let $G_\sigma$ be a signed graph.  It is antibalanced if, and only if, $G_{-\sigma}$ has only positive cycles.  The sign of a cycle is the same in $G_\sigma$ and $G_{-\sigma}$ if the cycle has even length and is the opposite if the cycle has odd length.  Thus, $G _{-\sigma}$ is balanced if, and only if, every even cycle in $G_\sigma$ is positive and every odd cycle in $G_\sigma$ is negative.
\end{proof}

\begin{definition}\label{6}{\rm\cite{zo}}
{\rm
For a biorientation $\tau$ of a graph $G_{\tau}  = (V, E; \tau)$, we
define a signature $\sigma$ of $E$, for an edge $e$ with ends $x$ and $y$, by:
$$\sigma(e) = - \tau(e, x)\tau(e,y).$$
}
\end{definition}

\begin{definition}\label{7}
{\rm
A signed or bidirected graph is \emph{all positive} (resp., \emph{all negative}) if all its edges are positive (resp., negative); i.e. in a bidirected graph, for every edge $e$, $\overline{W}(e) = 0$ (resp., $\neq0$).
}
\end{definition}

We observe that a bidirected graph that is all positive is a usual directed graph.

Each bidirected graph determines a unique signature.  However, the number of biorientations of a signed graph is $2^{|E|}$ because each edge has two possible biorientations.

\begin{definition}\label{8} {\rm\cite{k,zo}}
{\rm
Let $G_{\tau}  = (V, E; \tau)$ be a bidirected graph and let $X$ be a set of vertices of $G$.  
A new biorientation $\tau_{X}$  of $G$ is defined as follows:
\begin{align*}
\tau_{x}(e, x) &= - \tau(e, x), \ \forall\ x \in X,\\[6pt]
\tau_{x}(e, y) &= \tau(e, y), \ \forall\ y \in V - X,
\end{align*}
 for any edge $e \in E$, where $x$ and $y$ are the ends of the edge $e$.  We say that the biorientation $\tau_{X}$  and the bidirected graph $G_{\tau_{X}}$   are obtained respectively from $\tau$ and $G_{\tau}$ by \emph{switching} $X$.  If $X = \{x\}$ where $x\in V$, we write $\tau_{x}$ for simplicity.
The definition of switching a signed graph is similar.  Let $G_\sigma$ be a signed graph and $X \subseteq V$.  The sign function $\sigma$ \emph{switched by $X$} is $\sigma_X$ defined as follows:
\begin{gather*}
\sigma_X(e) = \begin{cases}
\sigma(e),	&\text{ if } x,y\in X \text{ or } x,y \in V-X, \\
-\sigma(e), &\text{ otherwise.}
\end{cases}
\end{gather*}
}
\end{definition}

We note that switching $X$ is a self-inverse operation.  It also follows from the definitions that the following result holds:

\begin{proposition}\label{ssw}
 Let $G_\tau$ be a bidirected graph and $\sigma$ the signature determined by $\tau$.   Let $X \subseteq V$.  Then $\tau_X$ determines the signature $\sigma_X$.
 \end{proposition}

\setlength{\unitlength}{.5mm}

\vspace*{0.5cm}
\begin{figure}[h]
\begin{center}
\begin{picture}(30,30)(0,0)
\multiput (0,0)(30,0){2}{\circle*{3}}
\multiput (0,30)(30,0){2}{\circle*{3}}
\multiput(0,0)(0,30){2}{\line(1,0){30}}
\multiput(0,0)(30,0){2}{\line(0,1){30}}
\put(2,-6.75){$+$}
\put(-9,4){$+$}
\put(32,4){$+$}
\multiput(2,31)(19,0){2}{$-$}
\put(20,-6.75){$-$}
\multiput(-7.5,20)(39,0){2}{\rotatebox{90}{$-$}}
\put(-10,-7){$d$}
\put(35,-6){$c$}
\put(-8,32){$a$}
\put(35,32){$b$}
\end{picture}
\hspace*{1cm}
\begin{picture}(20,20)(0,0)
\put(0,17){\vector(1,0){40}}
\put(0,21){\small{by switching}}
\put(10,9){\small{\{$b, c$\}}}
\end{picture}
\hspace*{2cm}
\begin{picture}(30,30)(0,0)
\multiput (0,0)(30,0){2}{\circle*{3}}
\multiput (0,30)(30,0){2}{\circle*{3}}
\multiput(0,0)(0,30){2}{\line(1,0){30}}
\multiput(0,0)(30,0){2}{\line(0,1){30}}
\put(2,-6.75){$+$}
\put(-9,4){$+$}
\put(32,22){$+$}
\put(20,-6.75){$+$}
\put(2,31){$-$}
\put(20,32){$+$}
\put(-7.5,20){\rotatebox{90}{$-$}}
\put(32,2){\rotatebox{90}{$-$}}
\put(-10,-7){$d$}
\put(35,-6){$c$}
\put(-8,32){$a$}
\put(35,32){$b$}
\end{picture}
\end{center}
\vspace{.2cm}
\caption{Example of switching a bidirected graph.}
\end{figure}
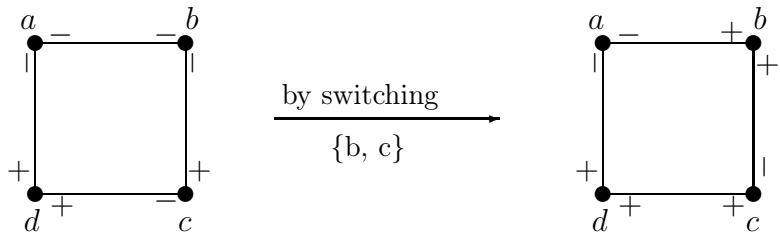

\begin{proposition}\label{9}
\begin{enumerate}[{\rm(i)}]
\item {\rm\cite{zs}} The result of switching a balanced signed graph is balanced.
\item {\rm\cite{hn,zs}} A signed graph is balanced if, and only if, there is a subset $X$ of vertices such that switching $X$ produces a signed graph in which all edges are positive.
\item {\rm\cite{hd,zs}} A signed graph is antibalanced if, and only if, there is a subset $X$ of vertices such that switching $X$ produces a signed graph in which all edges are negative.
\end{enumerate}
\end{proposition}

\begin{proof}
(i) Switching does not change the sign of any cycle.

(ii) Harary \cite{hn} has shown that the set of negative edges of a balanced
signed graph, if it is not empty, constitutes a cocycle of $G_{\sigma}$.
The cocycle divides $V$ into two sets, $X$ and $V-X$, such that $F$ consists of all edges with one end in each set.  
Thus by switching $X$, we obtain $\sigma(e) = +1 \ \forall\ e \in F$ in the new graph $G_{\sigma_{X}}$ and the other edge signs remain positive.  Thus $G_{\sigma_{X}}$ is all positive.  

Conversely, if there exists $X \subseteq V$ such that $G_{\sigma_{X}}$ is all positive, then $G_{\sigma_{X}}$ is balanced, so $G_\sigma$ is balanced by part (i).

(iii) $G_\sigma$ is antibalanced $\Leftrightarrow$ $G_{-\sigma}$ is balanced $\Leftrightarrow$ $\exists\ X \subseteq V$ such that $G_{(-\sigma)_{X}}=G_{-\sigma_{X}}$ is all positive $\Leftrightarrow$ $\exists\ X \subseteq V$ such that $G_{\sigma_X}$ is all negative.
\end{proof}

Proposition \ref{9} applies to bidirected graphs (cf.\ \cite{b}) because of Definition \ref{6} and Proposition \ref{ssw}.  Similarly, all propositions about signed graphs $G_\sigma$ apply to bidirected graphs $G_\tau$ through the signature $\sigma$ determined by $\tau$.

\subsection{Bipaths in Bidirected Graphs }

\begin{definition}\label{11} {\rm\cite{b}}
{\rm
Let $G_{\tau}$ be a bidirected graph, and let $P$ be a chain connecting $x$ and $y$ in $G_{\tau}$:
$P : x e_{1} x_{1} \ldots e_{i} x_{i} e_{i+1} \ldots x_{k-1} e_{k} y.$
We define 
$$W_{P}(x_{i})=\tau (e_{i}, x_{i})+\tau(e_{i+1}, x_{i}) \text{ for every } x_{i}\in V(P),\ i=1,\ldots,k-1.$$ 
(We note that $W_{P}(x_{i})$ and $W_{P}(x_{j})$ may differ when $i \neq j$, even if $x_i=x_j$.)  
Let $\tau(e_{1}, x) =\alpha$ and $\tau(e_{k},y)=\beta$; then we write
$$P=P_{(\alpha, \beta)}(x, y): x^{\alpha} e_{1} x_{1} \ldots e_{i} x_{i+1} e_{i+1} \ldots x_{k-1} e_{k} y^{\beta}.$$
We call $P_{(\alpha, \beta)}(x, y)$ an \emph{$(\alpha, \beta)$ bipath} from $x$ to $y$, or more simply 
a \emph{bipath from $x^{\alpha}$ to $y^{\beta}$}, if:
\begin{enumerate}[{\rm(i)}]
\item $k \geq 1$.
\item $\tau(e_{1}, x) =\alpha$, and $\tau(e_{k},y)=\beta$.
\item $W_{P}(x_{i}) = 0$, $\forall i = 1, \ldots, k - 1$ (if $k > 1$).
\item $P_{(\alpha, \beta)}(x, y)$ is minimal for the properties {\rm(i)--(iii)}, given $x^\alpha$ and $y^\beta$.
\end{enumerate}

If $P_{(\alpha, \beta)}(x, y)$ satisfies (i)--(iii), we call it a \emph{bichain} from $x^\alpha$ to $y^\beta$.  Thus a bipath is a minimal bichain (in the sense of (iv)); however, it need not be a path (an elementary chain).
}
\end{definition}

In the notation for a bipath $P$, we define $x_0 = x$ and $x_k = y$.  Then edge $e_i$ has vertices $x_{i-1}$ and $x_i$, $\forall i=1,2,\ldots,k$.

\begin{definition}\label{defreverse}
{\rm 
If $P_{(\alpha, \beta)}(x, y)$ is a bipath from $x^{\alpha}$  to $y^{\beta}$, then 
$$P_{(\beta, \alpha)}(y, x): y^{\beta} e_{k} x_{k-1} \ldots e_{i+1} x_{i} e_{i} \ldots x_{1} e_{1} x^{\alpha}$$ 
is also a bipath, from $y^{\beta}$ to $x^{\alpha}$.  It is called the \emph{reverse} of $P_{(\alpha, \beta)}(x, y)$.
}
\end{definition}

\begin{remark}\label{bpathconsec}
{\rm 
In a bipath no two consecutive edges $e_i, e_{i+1}$ can be equal, because then by cutting out $e_i x_i e_{i+1}$ we obtain a shorter bipath, which is absurd.
}
\end{remark}

\begin{proposition}\label{bwalksign}
{\rm 
The sign of a bichain $P_{(\alpha, \beta)}(x, y)$ is $\sigma(P) = -\alpha\beta$.
}
\end{proposition}

\begin{proof}
Let  $P_{(\alpha, \beta)}(x, y) : x^{\alpha}e_{1}x_{1}\ldots x_{k-1}e_{k}y^{\beta}$ be a bichain from $x^{\alpha}$ to $y^{\beta}$.  The sign of this bichain is given by
\begin{align*}
\sigma(P_{(\alpha, \beta)}(x, y))&= {\prod_{e\in P_{(\alpha, \beta)}(x, y)}\sigma(e)} \\
 &= [-\tau(e_{1},x)\tau(e_{1}, x_{1})][-\tau(e_{2}, x_{1})\tau(e_{2}, x_{2})] \ldots \\
&\qquad [-\tau (e_{k-2}, x_{k-2})\tau (e_{k-1}, x_{k-1})]
[-\tau(e_{k}, x_{k-1})\tau (e_{k}, y)] \\
 &= -\tau (e_{1}, x) [-\tau (e_{1}, x_{1})\tau(e_{2}, x_{1})] \ldots\\
 &\qquad [-\tau (e_{k-1}, x_{k-1}) \tau (e_{k},
x_{k-1})] \tau (e_{k}, y).
\end{align*}
According to the definition of bichains we have 
$$W_{P}(x_{i})=\tau (e_{i}, x_{i})+\tau(e_{i+1}, x_{i})= 0,$$
therefore $\tau(e_{i}, x_{i})\tau(e_{i+1}, x_{i}) = -1 \ \forall\ i = 1,\ldots,k - 1.$  Thus, 
$$
\sigma (P_{(\alpha, \beta)}) = - \tau(e_{1}, x) \tau(e_{k}, y) =  -\alpha\beta,
$$ 
which proves the result.
\end{proof}

\begin{proposition}\label{bpath}
Let $G_{\tau}$ be a bidirected graph, and let 
$$P : x_0 e_{1} x_{1} \ldots e_{i} x_{i} e_{i+1} \ldots x_{k-1} e_{k} x_k,$$
where $k\geq1$ and $e_{i}=\{x_{i-1}^{\alpha_{i-1}}, x_i^{\beta_i}\}$ for $i=1,\ldots,k$, be a chain in $G_\tau$.  
Then $P$ is a bipath if, and only if, 
\begin{enumerate}[{\rm(a)}]
\item $\alpha_{i}=-\beta_{i}$ for $i=1,\ldots,k-1$ and
\item $x_i^{\alpha_i} \neq x_j^{\alpha_j}$ when $i<j$ and $(i,j) \neq (0,k)$;
\end{enumerate}
and then it is an $(\alpha_0,\beta_k)$ bipath from $x_0^{\alpha_0}$ to $x_k^{\beta_k}$.
\end{proposition}

\begin{proof}
Let $x=x_0$, $y=x_k$, $\alpha = \alpha_0$, $\beta = \alpha_k$.  
Since $W_{P}(x_{i}) = \beta_i+\alpha_{i}$ for $i=1,\ldots,k-1$, condition (iii) is equivalent to condition (a).

Assume $P$ is an $(\alpha,\beta)$ bipath from $x$ to $y$.    Therefore, $P$ is a chain 
$$x^{\alpha_0} e_1 x_1 \ldots e_i x_i e_{i+1} \ldots x_{k-1} e_{k} y^{\beta_k}$$
from $x_0^{\alpha_0}$ to $x_k^{\beta_k}$ that satisfies (i)--(iii) in Definition \ref{11}.  
If $x_i^{\alpha_i} = x_j^{\alpha_j}$ for some $i<j$, then by cutting out $e_{i+1} \ldots e_{j}$ we get a shorter chain with the same properties (i)--(iii), unless $(i,j) = (0,k)$.  
We conclude that if $x_i=x_j$ ($i<j$ and $(i,j) \neq (0,k)$), then $\alpha_i\neq\alpha_j$.  
It follows that $P$ satisfies (b).

Assume $P$ satisfies (a) and (b).  Then it satisfies (i)--(iii).  Suppose $P$ were not minimal with those properties.  Then there is an $(\alpha,\beta)$ bipath $Q$ from $x$ to $y$ whose edges are some of the edges of $P$ in the same order as in $P$.  
If $Q$ begins with edge $e_{i+1}$, then it begins at $x_i^{\alpha_i}$ and $x_i^{\alpha_i} = x^\alpha = x_0^{\alpha_0}$, therefore $i=0$ by (b).  Similarly, $Q$ ends at edge $e_k$ and vertex $x_k^{\beta_k}$.  
If $Q$ includes edges $e_i$ and $e_{j+1}$ with $i<j$ but not edges $e_{i+1},\ldots,e_j$, then $x_i^{\alpha_i} = x_i^{-\beta_i} = x_j^{\alpha_j}$, contrary to (b).  
Therefore, $Q$ cannot omit any edges of $P$.  It follows that $P$ is minimal satisfying (i)--(iii), so $P$ is an $(\alpha,\beta)$ bipath from $x$ to $y$.  
\end{proof}

\begin{corollary}\label{positive cyclic bipath}
If $P$ is a bipath that contains a positive cycle $C$, then $P=C$.
\end{corollary}

Examples of bipaths can be seen in figure \ref{FIG5}.

We now give the different types of bipath which have a unique cycle that is negative.

\begin{definition}\label{12}
{\rm
A \emph{purely cyclic bipath} at a vertex $x$ in a bidirected graph is a bipath $C$ from $x$ to $x$ whose chain is a cycle. We say $C$ is \emph{on} the vertex $x$.  The sign of $C$ is the sign of its chain.
}
\end{definition}

We note that in a purely cyclic negative bipath $C$ on $x$, $x$ is the unique vertex in $V(C)$ such that $W_{C}(x) = \pm 2$.

\begin{definition}\label{13}
{\rm
A \emph{cyclic bipath} $P$ connecting two vertices $x$ and $y$ (not necessarily distinct) in a bidirected graph $G_{\tau}$, is a bipath from $x$ to $y$ which contains a unique purely cyclic bipath, which is negative.
Figure \ref{FIG5} shows the three possible cases.  We note that $\alpha, \beta, \gamma, \lambda  \in\{-1, +1\}.$  
If $x=y$ in type (a), the cyclic bipath is purely cyclic.
}
\end{definition}

\begin{figure}[h]
  \includegraphics[scale=.8]{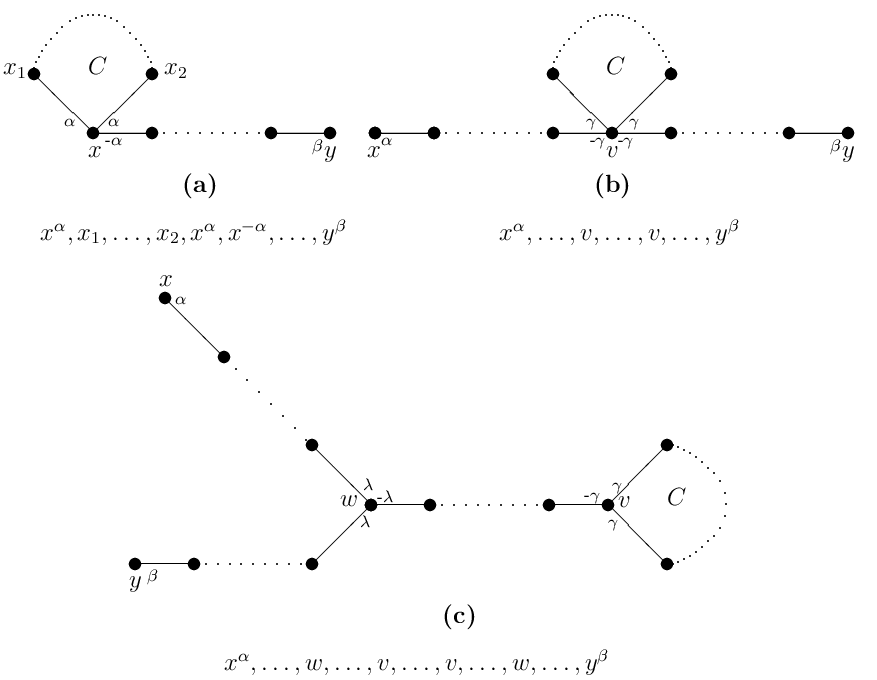}
  \\
  \caption{The three types of cyclic bipath.  In (a), $x=y$ is possible.  In (c), $x$ or $y$ or both may equal $w$, but $w \neq v$.}\label{FIG5}
\end{figure}

\begin{lemma}\label{cbpath}
A cyclic bipath must have one of the forms in figure \ref{FIG5}.
\end{lemma}

\begin{proof}
Let $P$ be a cyclic bipath from $x$ to $y$, $C$ the purely cyclic bipath in $P$, and $v$ the vertex at which $W_C(v)=\pm2$.  The graph of $P$ must consist of $C$ and trees attached to $C$ at a vertex, and it can have at most two vertices with degree 1 because $P$ has only two ends.  
We may assume $P \neq C$ and $y \neq v$.  
Since $W_C(v) \neq 0$, $P$ must enter $C$ at $v$ (unless $x=v$) and leave it at $v$ to get to $y$.  Therefore, there must be a tree attached to $v$.  
There cannot be a tree attached to any vertex $z$ of $C$ other than $v$, because $P$ would have to enter the tree from $C$ at $z^\gamma$ and retrace its path back to $z^\gamma$ in $C$, which would oblige $P$ to contradict Remark \ref{bpathconsec}.  Therefore $x$ and $y$ are both in the tree $T$ attached to $v$, possibly with $x=v$.  If $x\neq v$, then $x$ must be a vertex of degree 1 in $T$, or $P$ would contradict Remark \ref{bpathconsec}.  Similarly, $y$ must be a vertex of degree 1 in $T$.  As $T$ must be the union of the paths in $T$ from $x$ and $y$ to $v$, $P$ can only be one of the types in figure \ref{FIG5}.
\end{proof}

\begin{definition}\label{14}{\rm\cite{b}}
{\rm
Let $G_{\tau}$ be a bidirected graph, let $\alpha, \beta \in \{-1, +1\}$, and let $C$ be a bipath $C: x^{\alpha}\ e_{1}\ x_{1}\ \ldots\ x_{k-1}\ e_{k}\ x^{\beta}.$
If $\alpha = -\beta $, we say that $C$ is a \emph{bicircuit} of $G_{\tau}$.
}
\end{definition}

\setlength{\unitlength}{.4mm}

\vspace*{1.0cm}
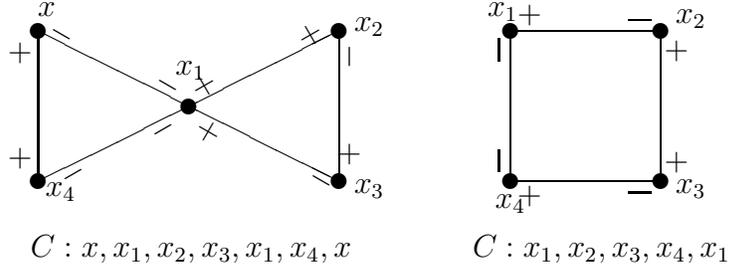
\begin{figure}[h]
\begin{center}
\begin{picture}(80,40)(0,-15)
\put(0,0){\line(0,1){40}}
\put(0,0){\line(2,1){80}}
\put(0,40){\line(2,-1){80}}
\put(80,0){\line(0,1){40}}
\multiput(0,0)(40,20){3}{\circle*{4}}
\multiput(0,40)(80,-40){2}{\circle*{4}}
\multiput(42,25)(24,12){2}{\rotatebox{30}{$+$}}
\multiput(-9,5)(0,25){2}{$+$}
\put(82,2){\rotatebox{90}{$+$}}
\put(42,9){\rotatebox{-30}{$+$}}
\multiput(4,41)(24,-12){2}{\rotatebox{-30}{$-$}}
\put(81,30){\rotatebox{90}{$-$}}
\multiput(4,-5)(24,12){2}{\rotatebox{30}{$-$}}
\put(65,0){\rotatebox{-30}{$-$}}
\put(-8,-10){$x_{4}$}
\put(80,-10){$x_{3}$}
\put(-6,45){$x$}
\put(80,45){$x_{2}$}
\put(36,35){$x_{1}$}
\put(-12,-25){$C : x, x_{1}, x_{2}, x_{3}, x_{1}, x_{4}, x$}
\end{picture}
\hspace*{3cm}
\begin{picture}(40,40)(0,-15)
\multiput(0,0)(0,40){2}{\line(1,0){40}}
\multiput(0,0)(40,0){2}{\line(0,1){40}}
\multiput(0,0)(40,0){2}{\circle*{4}}
\multiput(0,40)(40,0){2}{\circle*{4}}
\multiput(-7,4)(0,25){2}{\rotatebox{90}{$-$}}
\put(28,-6){$-$}
\put(28,42){$-$}
\multiput(42,3)(0,30){2}{$+$}
\put(5,-7){$+$}
\put(5,43){$+$}
\put(-8,-10){$x_{4}$}
\put(40,-10){$x_{3}$}
\put(40,46){$x_{2}$}
\put(-8,46){$x_{1}$}
\put(-20,-25){$C : x_{1}, x_{2}, x_{3}, x_{4}, x_{1}$}
\end{picture}
\end{center}
\caption{Two kinds of bicircuit.}\label{FIG4}
\end{figure}

%333333333333333333333333333333333333333333333333333333333333333333333333333333333333333333333333333333333333333333333333333333333333333333333
\section{Transitive Closure in Bidirected Graphs}

\begin{definition}\label{15t}
Let $G_{\tau}$ be a bidirected graph.  $G_\tau$ is \emph{transitive} if, for any vertices $x$ and $y$ (not necessarily distinct) such that there is an $(\alpha,\beta)$ bipath from $x^\alpha$ to $y^\beta$ in $G_\tau$, there is an edge $\{x^\alpha, y^\beta\}$ in $G_\tau$.
\end{definition}

\begin{definition}\label{15}
{\rm
Let $G_{\tau}$ be a bidirected graph.
The \emph{transitive closure} of  $G_{\tau}$ is the graph, notated $\Ft(G_{\tau}) = (V, \Ft(E); \tau)$, such that $\{x^{\alpha}, y^{\beta}\} \in\Ft(E)$ if there is a bipath $P_{(\alpha, \beta)}(x, y)$ from
$x^{\alpha}$ to $y^{\beta}$ in $G_{\tau}$ ($x$ and $y$ are not necessarily distinct).
}
\end{definition}

\begin{remark}\label{transitive closure edges}
{\rm 
We see that $E \subseteq \Ft(E)$.  If $\{x^{\alpha}, y^{\beta}\} \in E$, then $\{x^{\alpha}, y^{\beta}\}$ is the edge of a bipath of length $1$, so $\{x^{\alpha}, y^{\beta}\} \in\Ft(E)$.
}
\end{remark}

\begin{remark}\label{transitive closure loops}
{\rm 
If there is a bipath $P_{(\alpha, \beta)}(x, x)$ from $x^{\alpha}$ to $x^{\beta}$ in $G_{\tau}$, then there is a loop $\{x^\alpha,x^\beta\}$ with sign $-\alpha\beta$ in $\Ft(G_{\tau})$.
}
\end{remark}

\begin{remark}\label{transitive closure bicircuits}
{\rm 
If $G_\tau$ contains a bicircuit $C$, then $\Ft(G_\tau)$ contains all the edges $\{x^{-\alpha},y^{-\beta}\}$ such that $\{x^{\alpha},y^{\beta}\} \in E(C)$.  In other words, the transitive closure contains the opposite orientation of every edge that lies in a bicircuit in $G_\tau$.  For an example see figure \ref{FIG7}.
}
\end{remark}

\begin{proposition}\label{ft2}
$\Ft$ is an abstract closure operator; that is: 
\begin{enumerate}[{\rm(i)}]
\item $G_{\tau}$ is a partial graph of $\Ft(G_{\tau})$.
\item If $H_\tau$ is a partial graph of $G_\tau$, then $\Ft(H_\tau)$ is a partial graph of $\Ft(G_\tau)$.
\item $\Ft(\Ft(G_{\tau})) = \Ft(G_{\tau})$.
\end{enumerate}
\end{proposition}

\begin{proof}
(i) and (ii) are obvious from the definition.

(iii) We prove that $\Ft(G_\tau)$ is transitive.  Let $P: x e_{1} x_{1} \ldots e_{i} x_{i} e_{i+1} \ldots x_{k-1} e_{k} y$ be a bipath in $\Ft(G_{\tau})$, with $e_{i}=\{x_{i-1}^{\alpha_{i-1}}, x_i^{\beta_i}\}$ for $i=1,\ldots,k$.  By Proposition \ref{bpath} $\alpha_{i}=-\beta_{i}$ for $i=1,\ldots,k-1$ and $P$ is an $(\alpha_0,\beta_k)$ bipath.  For each edge $e_i$ there is an $(\alpha_{i-1},\beta_i)$ bipath $Q_i(x_{i-1}, x_i)$ in $G_{\tau}$ (which may be $e_i$ itself).  Let $R = x_0 Q_1 x_1 Q_2 \ldots Q_k x_k$, the concatenation of $Q_1,\ldots,Q_k$.  
At each intermediate vertex $z$ of any $Q_i$ we have $W_{R}(z) = W_{Q_i}(z) = 0$ by property (iii) of Definition \ref{11}.  At each $x_i$, $i=1,\ldots,k-1$ we have $W_{R}(x_i) = \beta_i+\alpha_i = 0$ by Proposition \ref{bpath}.  Therefore, $R$ is an $(\alpha_0,\beta_k)$ bipath from $x$ to $y$ in $G_{\tau}$.  We deduce that the edge $\{x^{\alpha_0}, y^{\beta_k}\}$ is in the transitive closure of $G_{\tau}$.  Thus, $\Ft(G_{\tau})$ is transitive and its transitive closure is itself.
\end{proof}

Define  
\begin{equation*}
\overline{W}(P_{(\alpha, \beta)}(x, y)) = \sum_{e\in {P_{(\alpha, \beta)}(x, y)}}\overline{W}(e).
%\label{Wpath}
\end{equation*}

\begin{theorem}\label{16}
Given a bidirected graph $G_{\tau}$ and its transitive closure $\Ft(G_{\tau}) = (V, \Ft(E); \tau)$.  If $e =\{x^{\alpha}, y^{\beta}\}$ is the edge in $\Ft(G_{\tau})$ implied by transitive closure of the bipath $P_{(\alpha, \beta)}(x, y)$ from $x^{\alpha}$ to $y^{\beta}$ in $G_{\tau}$, then 
$$\overline{W}(P_{(\alpha, \beta)}(x, y)) = \overline{W}(e).$$
\end{theorem}

\begin{proof}
Let $P_{(\alpha, \beta)}(x, y) : x^{\alpha}e_{1}x_{1}\ldots x_{k-1}e_{k}y^{\beta}$  be a bipath from $x^{\alpha}$ to $y^{\beta}$  in $G_{\tau}$.  According to Definition \ref{2} we have 
$$\overline{W}(P_{(\alpha, \beta)}(x, y)) = \tau(e_{1},x) + W(x_{1}) + W(x_{2}) + \ldots + W(x_{k-1}) + \tau(e_{k},y),$$
and by Definition \ref{11} we obtain
$$\overline{W}(P_{(\alpha, \beta)}(x, y)) =  \tau(e_{1},x) + \tau(e_{k},y) = \alpha + \beta.$$
According to Definition \ref{2}, $\overline{W}(e) = \alpha + \beta = \overline{W}(P_{(\alpha, \beta)})$.
\end{proof}

We recall that $\sigma(P)$ designates the sign of a chain $P$ (Definition \ref{10}).

\begin{corollary}\label{17}
Given a bidirected graph $G_{\tau}$ and its transitive closure $\Ft(G_{\tau}) = (V, \Ft(E); \tau)$. If $e =\{x^{\alpha},  y^{\beta}\}$ is the edge implied by transitive closure of the bipath $P_{(\alpha, \beta)}(x, y)$ from $x^{\alpha}$
to $y^{\beta}$, then 
$$\sigma(P_{(\alpha, \beta)}(x, y))=\sigma(e).$$
\end{corollary}
\begin{proof}
The sign of the bipath is $\sigma(P_{(\alpha, \beta)}(x, y)) = -\alpha\beta$ by Proposition \ref{bwalksign}.
We have $ \sigma(e) = - \tau (e, x)  \tau(e, y) = -\alpha\beta,$ 
from which the result follows.
\end{proof}

\begin{figure}[h]
\begin{center}
\begin{picture}(180,10)(0,0)
\multiput (0,0)(60,0){4}{\circle*{4}}
\put(0,0){\line(1,0){180}}
\put(-2,-12){$x$}
\put(59,-12){$a$}
\put(119,-12){$b$}
\put(179,-12){$y$}
\multiput(2,3)(45,0){2}{$-$}
\multiput(64,3)(45,0){2}{$+$}
\multiput(124,3)(45,0){2}{$-$}
\put (37,-30){$P_{(-,-)}(x,y) : x^{-},  a,  b,  y^{-}$}
\end{picture}
\vspace*{2cm}

\begin{picture}(180,20)(0,0)
\multiput (0,0)(60,0){4}{\circle*{4}}
\put(0,0){\line(1,0){180}}
\put(-3,-12){$x$}
\put(59,-12){$a$}
\put(119,-12){$b$}
\put(179,-12){$y$}
\curve(0,0, 60,18, 120,0)
\curve(0,0, 90,-18, 180,0)
\curve(60,0, 120,18, 180,0)
\multiput(15,1)(32,0){2}{$-$}
\multiput(64,-7)(45,0){2}{$+$}
\multiput(125,1)(32,0){2}{$-$}
\put(1,3){\rotatebox{30}{$-$}}
\put(110,8){\rotatebox{-30}{$+$}}
\put(60,5){\rotatebox{30}{$+$}}
\put(168,7.5){\rotatebox{-28}{$-$}}
\put(1,-6){\rotatebox{-24}{$-$}}
\put(167,-3){\rotatebox{-156}{$-$}}
\put(55,-35){$\Ft(P_{(-,-)}(x,y))$}
\end{picture}
\end{center}
\vspace{1.3cm}
\caption{An example of transitive closure of a bidirected graph.}\label{FigExplan}
\end{figure}
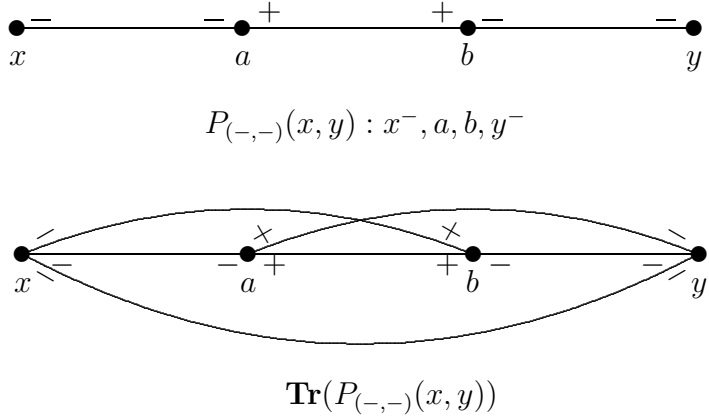

\begin{lemma}\label{18}
Let $G_\tau$ be a bidirected graph and $X\subseteq V$.  Then $\Ft(G_{\tau_X})$ is the result of switching $\Ft(G_\tau)$ by $X$.
\end{lemma}

\begin{proof}
We observe that a bipath in $G_\tau$ remains a bipath after switching $G_\tau$.
For a set $X \subseteq V$, define $\iota(v)=+1$ if $v\notin X$ and $-1$ if $v\in X$.

Assume that $e=\{x^\alpha,y^\beta\}$ is an edge of $\Ft(G_\tau)$, not in $E(G_\tau)$, that is implied by a bipath $P_{(\alpha, \beta)}(x, y)$ in $G_\tau$.  Switch $G_\tau$ and $\Ft(G_\tau)$ by $X$.  Then $e$ becomes $\{x^{\alpha\iota(x)},y^{\beta\iota(y)}\}$ and  $P_{(\alpha, \beta)}(x, y)$ becomes $P_{(\alpha\iota(x), \beta\iota(y))}(x, y)$.  Therefore $e$ is implied by the bipath $P_{(\alpha\iota(x), \beta\iota(y))}(x, y)$ in $G_{\tau_X}$, so $e$ is an edge in $\Ft(G_{\tau_X})$.  This proves that $\Ft(G_\tau)$ switched by $X$ is a partial graph of $\Ft(G_{\tau_X})$.

By similar reasoning, if $e=\{x^\alpha,y^\beta\}$ is an edge of $\Ft(G_{\tau_X})$ not in $E(G_{\tau_X})$, it is implied by a bipath $P_{(\alpha, \beta)}(x, y)$ in $G_{\tau_X}$.  Switching by $X$, the edge $\{x^{\alpha\iota(x)},y^{\beta\iota(y)}\}$ is implied by $P_{(\alpha\iota(x), \beta\iota(y))}(x, y)$, which is a bipath in $G_{\tau_X}$, so that $\{x^{\alpha\iota(x)},y^{\beta\iota(y)}\}$ is an edge of $G_{\tau_X}$ switched by $X$, which is $G_\tau$.  Therefore $\Ft(G_{\tau_X})$ switched by $X$ is a partial graph of $\Ft(G_\tau)$.  The result follows.
\end{proof}

\begin{proposition}\label{19}
The transitive closure of an all-positive bidirected graph is all positive.
The transitive closure of a balanced bidirected graph is balanced.
\end{proposition}

\begin{proof}
Assume $G_\tau$ is all positive.  Let $P_{(\alpha, \beta)}(x, y) : x^{\alpha}e_{1}x_{1}\ldots x_{k-1}e_{k}y^{\beta}$ be a bipath from $x^{\alpha}$ to $y^{\beta}$; we close this bipath by the positive edge $e=\{x^{\alpha}, y^{\beta}\}$.
Since $\overline{W}(e_i)=0$ for a positive edge $e_i$, $\overline{W}(P_{(\alpha, \beta)}(x, y))=0$.  By Theorem \ref{16}, $\overline{W}(e)=0$, which means that $\beta=-\alpha$.  Thus, $e$ is positive.

Assume $G_\tau$ is balanced.  By Proposition \ref{ssw} there is a vertex set $X \subseteq V$ such that $G_{\tau_X}$ is all positive.  By the first part, $\Ft(G_{\tau_X})$ is all positive, therefore balanced, and by Lemma \ref{18} it equals $\Ft(G_\tau)$ switched by $X$.  Therefore $\Ft(G_\tau)$ equals $\Ft(G_{\tau_X})$ switched by $X$, which is balanced by Proposition \ref{9}.
\end{proof}

The diagram below, obtained from the results above, shows that  the classical notion of transitive closure for directed graphs is a particular case of that found for bidirected graphs.
\vspace*{.5cm}
\begin{figure}[h]
\begin{center}
\begin{picture}(240,60)(-18,-15)
\put(-20,-13){{$G$ digraph}}
\put(30,-10){\vector(1,0){130}}
\put(30,52){\vector(1,0){130}}
\put(30,20){\vector(1,0){130}}
\put(67,-18) {{transitive closure}}
\put(165,-13){{$\Ft(G)$ digraph}}
\put(-30,50){{$G_{\tau}$ balanced}}
\put(67,44){{transitive closure}}
\put(67,12){{transitive closure}}
\put(165,50){{$\Ft(G_{\tau})$ balanced}}
\put(-26,18){{$G_{\tau}$ positive}}
\put(10,16){\vector(0,-1){20}}
\put(10,47){\vector(0,-1){20}}
\put(175,16){\vector(0,-1){20}}
\put(165,18){{$\Ft(G_{\tau})$ positive}}
\put(175,47){\vector(0,-1){20}}
\put(-48,35){{{by switching}}}
\put(180,35){{{by switching}}}
\put(-54,4){{{by orientation}}}
\put(180,4){{{by orientation}}}
\end{picture}
\end{center}
%\caption{\emph{Traitement de l'information}}%\label{}
\end{figure}

%444444444444444444444444444444444444444444444444444444444444444444444444444444444444444444444444444444444444444444444444444444444444444444444444

\section{Transitive Reduction in Bidirected Graphs}

\subsection{Definition and Basic Results}

\begin{definition}\label{20}
{\rm
Let $G_{\tau}=(V, E;\tau)$ be a bidirected graph.  Define $\ft(G_{\tau}; H_{\tau}) = \ft(H_{\tau})$ = the transitive closure of $H_{\tau}$ in $G_{\tau}$, where $H_{\tau}$ is a partial graph of $G_{\tau}$.  (We can write only $H_{\tau}$ when the larger graph, here $G_{\tau}$, is obvious.)
}
\end{definition}

\begin{proposition}\label{20a}
Let $H_\tau$ be a partial graph of $G_\tau$.  Then $$\ft(G_{\tau}; H_{\tau}) = \Ft(H_\tau) \cap G_\tau.$$
\end{proposition}

\begin{proof}
The definition implies that $\Ft(H_\tau) \cap G_\tau \subseteq \ft(G_{\tau}; H_{\tau})$.  

Let $e$ be an edge of $\ft(G_{\tau}; H_{\tau})$ not in $H_\tau$.  The edge $e$ is induced by a bipath $P$ in $H_\tau$.  Thus, $e\in \Ft(E(H_\tau))$.  It follows that $e$ is an edge of $\Ft(H_\tau)$.  Since also $e \in E(G_\tau)$, we conclude that $\ft(G_{\tau}; H_{\tau}) \subseteq \Ft(H_\tau) \cap G_\tau$.
\end{proof}

\begin{definition}\label{21}
{\rm
Let $G_{\tau}=(V, E;\tau)$ be a bidirected graph. A \emph{transitive reduction} of $G_{\tau}$ is a minimal generating set under $\ft$. Thus we define $\Rt(G_{\tau})=(V,\Rt(E); \tau)$ to be a minimal partial graph of $G_{\tau}$ with the property that $\ft(G_{\tau}; \Rt(G_{\tau})) = G_{\tau}$.
We note that $\Rt(G_\tau)$ may not be unique; see Remark \ref{rm2}.
}
\end{definition}

The definitions and Proposition \ref{ft2} immediately imply that $$\ft(\Ft(G_{\tau}); \Rt(G_{\tau})) = \ft(\Ft(G_{\tau}); G_{\tau}) = \Ft(G_{\tau}).$$

\begin{proposition}\label{22}
Let $G_{\tau}$ be a bidirected graph and $\Rt(G_{\tau})$ a transitive reduction.  Then $\Ft(\Rt(G_{\tau}))= \Ft(G_{\tau})$.
\end{proposition}
\begin{proof}
$\Ft(\Rt(G_{\tau}))= \ft(\Ft(G_{\tau}); \Rt(G_{\tau})) = \Ft(G_{\tau})$.
\end{proof}

\begin{proposition}\label{23}
Let $G_{\tau}$ be a bidirected graph.  A partial graph $H_{\tau}$ of $G_{\tau}$ is a transitive reduction $\Rt(G_{\tau})$ if, and only if, it is minimal such that $G_{\tau}\subseteq \Ft(H_{\tau})$.
\end{proposition}
\begin{proof}
It follows from Proposition \ref{20a} that for $H_\tau$ to be a transitive reduction of $G_\tau$ it is necessary that $G_\tau \subseteq \Ft(H_\tau)$.  It follows that $H_\tau$ is a transitive reduction of $G_\tau$ $\Leftrightarrow$ $H_\tau$ is minimal such that $G_{\tau}\subseteq \Ft(H_{\tau})$.
\end{proof}

\begin{proposition}\label{26}
If $G_{\tau}$ is a connected bidirected graph, then $\Rt(G_{\tau})$ is also connected.
\end{proposition}
\begin{proof}
The operator $\Ft$ does not change the connected components of a graph.
\end{proof}

\begin{corollary}\label{28}
Let $G_{\tau}$ be a bidirected graph without positive loops.  If it has no bipath of length greater than 1, then $\Rt(G_{\tau})=G_{\tau}$ and every vertex is a source or a sink.
\end{corollary}
\begin{proof}
Since $\Rt(G_{\tau})$ is a partial graph  of $G_{\tau}$, it is enough to prove that each edge $e$ in $G_{\tau}$ is in $\Rt(G_{\tau})$. Assume that there exists an edge $e=\{x^{\alpha},y^{\beta}\} \in G_{\tau} - E(\Rt(G_{\tau}))$. According to the definitions of transitive reduction and transitive closure, there exists a bipath from $x^{\alpha}$ to $y^{\beta}$ in  $G_{\tau}-\{e\}$ with length $k\geq 2$, which is absurd.

If a vertex $x$ is neither a source nor a sink, it has incident half-edges $(e,x)$ and $(f,x)$ with $\tau(e,x)=+1$ and $\tau(f,x)=-1$.  Then $ef$ is a bipath of length 2, which is absurd, or $e=f$, which implies that $e$ is a positive loop, which is also absurd.
\end{proof}

We can characterize the graphs in Corollary \ref{28} as follows (see figure \ref{FIG8}):
\begin{itemize}
\item $G_{\tau} = (V_{+1} \cup V_{-1}, E; \tau)$.  ($V_{+1}$ and $V_{-1}$ are the sets of sources and sinks; see Definition \ref{1ssi}.)
\item $G_\tau$ is antibalanced (Definition \ref{5}).  (Thus, if $G_\tau$ is balanced, then it is bipartite.)  
The edges connecting a vertex of $V_{+1}$ to a vertex of  $V_{-1}$ are positive.
The edges connecting two vertices of the same set are negative.
\end{itemize}

\setlength{\unitlength}{.8mm}

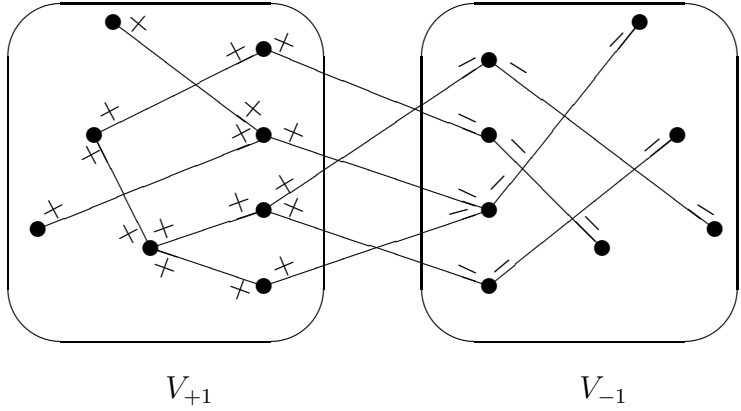
\begin{figure}[t]
\begin{center}
\begin{picture}(90,55)(0,-15)
\multiput(30,0)(0,10){3}{\circle*{2}}
\put(30,31.5){\circle*{2}}
\multiput(60,0)(0,10){4}{\circle*{2}}
\multiput(15,5)(60,0){2}{\circle*{2}}
\multiput(0,7.5)(90,0){2}{\circle*{2}}
\put(7.5,20){\circle*{2}}
\put(85,20){\circle*{2}}
\multiput(10,35)(70,0){2}{\circle*{2}}
\put(30,0){\line(-3,1){15}}
\put(15,5){\line(3,1){15}}
\put(15,5){\line(-1,2){7.5}}
\put(0,7.5){\line(5,2){30}}
\put(7.5,20){\line(2,1){22.5}}
\put(10,35){\line(4,-3){20}}
\put(30,31.5){\line(5,-2){30}}
\put(30,10){\line(3,2){30}}
\multiput(30,20)(0,-10){2}{\line(3,-1){30}}
\put(30,0){\line(3,1){30}}
\put(60,30){\line(4,-3){30}}
\put(60,20){\line(1,-1){15}}
\put(60,10){\line(4,5){20}}
\put(60,0){\line(5,4){25}}
\put(17,15){\oval(42,45)}
\put(72,15){\oval(42,45)}
\multiput(0,8.5)(25,10){2}{\rotatebox{28}{$+$}}
\multiput(5.1,18.25)(5,-10.4){2}{\rotatebox{-62}{$+$}}
\multiput(14.75,6)(10,3.5){2}{\rotatebox{20}{$+$}}
\multiput(14.9,2)(10,-3){2}{\rotatebox{-20}{$+$}}
\multiput(7.5,21.5)(16.75,8){2}{\rotatebox{28}{$+$}}
\multiput(11.25,35)(15.25,-11.25){2}{\rotatebox{-40}{$+$}}
\multiput(32,10)(0,10){2}{\rotatebox{-23}{$+$}}
\put(30.7,32){\rotatebox{-23}{$+$}}
\put(30.7,1){\rotatebox{23}{$+$}}
\put(30.7,11.75){\rotatebox{31}{$+$}}
\multiput(55,2)(0,10){2}{\rotatebox{-25}{$-$}}
\put(54,11.25){\rotatebox{195}{$-$}}
\put(55,22){\rotatebox{-25}{$-$}}
\put(55,28){\rotatebox{28}{$-$}}
\multiput(61.75,29)(24.9,-19){2}{\rotatebox{-32}{$-$}}
\multiput(61.75,19)(9.75,-10.25){2}{\rotatebox{-45}{$-$}}
\multiput(59,11.75)(15.8,19){2}{\rotatebox{46}{$-$}}
\multiput(59.5,0.75)(20,16){2}{\rotatebox{40}{$-$}}
\put(14,-15){$V_{+1}$}
\put(70,-15){$V_{-1}$}
\end{picture}
\end{center}
\caption{The form of a graph that satisfies the hypotheses of Corollary \ref{28}.}\label{FIG8}
\end{figure}

\begin{theorem}\label{29}
\begin{enumerate}[{\rm(i)}]
\item The transitive reduction $\Rt(G_{\tau})$ is balanced if, and only if, $G_{\tau}$ is balanced.
\item $\Rt(G_{\tau})$ is all positive if, and only if, $G_{\tau}$ is all positive.
\end{enumerate}
\end{theorem}

\begin{proof}
We conclude from Proposition \ref{19} that $\Rt(G_{\tau})$ is all positive (resp., balanced) $\Leftrightarrow$ $\Ft(\Rt(G_{\tau}))$ is all positive (resp., balanced) and that $G_{\tau}$ is all positive (resp., balanced) $\Leftrightarrow$ $\Ft(G_{\tau})$ is all positive (resp., balanced).  By Proposition \ref{22}, $\Ft(\Rt(G_{\tau})) = \Ft(G_{\tau})$.  The result follows.
\end{proof}

Theorem \ref{29}(ii) is important because all-positive bidirected graphs are the usual directed graphs.  Thus, it says that the transitive reduction of a directed graph, in our definition of transitive reduction, is a directed graph.
The diagram below, obtained from the results above, shows the stronger statement that the classical notion of transitive reduction for directed graphs is a particular case of our notion for bidirected graphs.

\setlength{\unitlength}{.4mm}

\vspace*{1.3cm}
\begin{figure}[h]
\begin{center}
\begin{picture}(200,50)(0,-15)
\put(-28,-13){{$G$ digraph}}
\put(23,-10){\vector(1,0){125}}
\put(47,-20) {{transitive reduction}}
\put(152,-13){{$\Rt(G)$ digraph}}
\put(-36,60){{$G_{\tau}$ balanced}}
\put (23,24){\vector(1,0){125}}
\put(47,12){{transitive reduction}}
\put (23,62){\vector(1,0){125}}
\put(47,50){{transitive reduction}}
\put(152,60){{$\Rt(G_{\tau})$ balanced}}
\put(10,20){\vector(0,-1){22}}
\put(-33,22){{$G_{\tau}$ positive}}
\put(10,55){\vector(0,-1){22}}
\put(155,20){\vector(0,-1){22}}
\put(152,22){{$\Rt(G_{\tau})$ positive}}
\put(155,55){\vector(0,-1){22}}
\put(-47,42){{{by switching}}}
\put(162.5,42){{{by switching}}}
\put(-54,5){{{by orientation}}}
\put(162.5,5){{{by orientation}}}
\end{picture}
\end{center}
\end{figure}

\begin{definition}\label{RE}
{\rm
Let $\RE(G_\tau)$ be the set of edges $e$ such that $e$ is in the transitive closure of $G_\tau-\{e\}$.  
We can say that these edges are \emph{redundant edges} in $G_{\tau}$.  
}
\end{definition}

\begin{lemma}\label{partialreduction}
If $H_\tau$ is a partial graph of $G_\tau$, then $\RE(H_\tau) \subseteq \RE(G_\tau)$.
\end{lemma}

\begin{proof}
If $e \in \RE(H_\tau)$, then $e \in \RE(G_\tau)$ by the definition of $\RE$.
\end{proof}

Note that two parallel edges (Definition \ref{1a}) are both redundant.  Thus, it is necessary to exclude parallel edges in the following proposition.

\begin{proposition}\label{Rt no bicircuits}
For a bidirected graph without bicircuits and without parallel edges, the graph $\Rt(G_{\tau})$ is unique.
It is obtained from $G_{\tau}$ by removing every redundant edge.  
\end{proposition}

\begin{proof}
Since there are no parallel edges, $e \in \RE(G_\tau)$ $\Leftrightarrow$ $e$ is in the transitive closure of a bipath of length at least $2$.  

We prove first that, if $e, f \in \RE(G_\tau)$, then $f \in \RE(G_\tau - \{e\})$.  Suppose $e=\{x^\alpha,y^\beta\}$ is implied by a bipath $P_0=P_{(\alpha,\beta)}(x,y)$ and $f=\{z^\gamma,w^\delta\}$ is implied by a bipath $Q_0=Q_{(\gamma,\delta)}(z,w)$.  

If $e \notin Q_0$, then $Q_0$ is a bipath in $G_\tau - \{e\}$ that implies $f$.  

If $e \in Q_0$ but $f \notin P_0$, then $Q_0 = Q_1 e Q_2$ where, by choice of notation, $e$ appears as $(x^\alpha,y^\beta)$ in that order, so that $Q_1 = P_{(\gamma,-\alpha)}(z,x)$ and $Q_2 = P_{(-\beta,\delta)}(y,w)$.  Replace $Q_0$ by $Q_1 P_0 Q_2$.  This is a bichain from $z^\gamma$ to $w^\delta$ so it contains a bipath $P$ from $z^\gamma$ to $w^\delta$, in $G_\tau - \{e\}$, and $P$ implies $f$.

If $e \in Q_0$ and $f \in P_0$, then $Q_0 = Q_1 e Q_2$ where $e$, $Q_1$ and $Q_2$ are as in the previous case, and $P_0 = P_1 f P_2$ where $f$ appears in $P_0$ as either $(z^\gamma,w^\delta)$ or $(w^\delta,z^\gamma)$.  
Suppose the first possibility.  Then $P_1 = P_{(\alpha,-\gamma)}(x,z)$ so $P_1Q_1$ is a bichain from $x^\alpha$ to $x^{-\alpha}$; therefore $P_1Q_1$ contains a bicircuit, which is impossible.  
Now suppose the second possibility and let $P^*$ denote the reverse of the bipath $P$ (Definition \ref{defreverse}).  Then $P_1 Q_1 P_2^* Q_2^*$ is a bichain from $x^\alpha$ to $x^{-\alpha}$; therefore it contains a bicircuit, which is impossible.  Therefore, this case cannot occur.

We conclude that $f \in \RE(G_\tau - \{e\})$ for every edge $f \in \RE(G_\tau)$, $f \neq e$.  Therefore, $\RE(G_\tau - \{e\}) \supseteq \RE(G_\tau) - \{e\}$.  
Since $G_\tau - \{e\}$ is a partial graph of $G_\tau$, $\RE(G_\tau) - \{e\}) \subseteq \RE(G_\tau)$ so $\RE(G_\tau - \{e\}) = \RE(G_\tau) - \{e\}$.  
By induction, $\RE(G_\tau - \RE(G_\tau)) = \RE(G_\tau) - \RE(G_\tau) = \emptyset$.  
We also conclude that $f \in \Ft(G_\tau - \{e\})$ and by induction that $\RE(G_\tau) \subseteq \Ft(G_\tau - \RE(G_\tau))$.  
Therefore, $\Rt(G_\tau) = G_\tau - \RE(G_\tau)$.  
This is unique.
\end{proof}

Figure \ref{FIG6} shows that the transitive closure of the bipath $P_{(-,-)}(2,3) : 2^{-},1,3^{-}$ contains the edge $\{2^{-}, 3^{-}\}$ which is a redundant edge.

\setlength{\unitlength}{.8mm}

\begin{center}
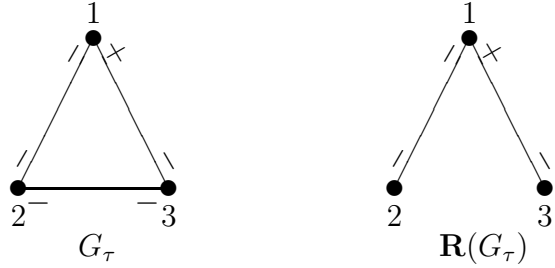
\begin{figure}[h]
\begin{picture}(80,28)(0,0)
\multiput(0,0)(20,0){2}{\circle*{2}}
\put(10,20){\circle*{2}}
\put(0,0){\line(1,0){20}}
\put(0,0){\line(1,2){10}}
\put(20,0){\line(-1,2){10}}
\put(-1,-5){$2$}
\put(19,-5){$3$}
\put(9,22){$1$}
\multiput(1,-3)(14.5,0){2}{\rotatebox{0}{$-$}}
\multiput(-1.5,2)(6.8,14){2}{\rotatebox{60}{$-$}}
\put(11,19){\rotatebox{-60}{$+$}}
\put(18,4.5){\rotatebox{-60}{$-$}}
\put(8,-9){$G_{\tau}$}
\hspace*{5cm}
\multiput(0,0)(20,0){2}{\circle*{2}}
\put(10,20){\circle*{2}}
\put(0,0){\line(1,2){10}}
\put(20,0){\line(-1,2){10}}
\put(-1,-5){$2$}
\put(19,-5){$3$}
\put(9,22){$1$}
\multiput(-1.5,2)(6.8,14){2}{\rotatebox{60}{$-$}}
\put(11,19){\rotatebox{-60}{$+$}}
\put(18,4.5){\rotatebox{-60}{$-$}}
\put(6,-9){$\Rt(G_{\tau})$}
\end{picture}
\vspace{.5cm}
\caption{$\{2^{-},3^{-}\}$ is a redundant edge.}\label{FIG6}
\end{figure}
\end{center}

\begin{remark}\label{rm2}
{\rm 
We show an example in which the transitive reduction is unique, and an example in which it is not unique. 
If $C_{1}$ and $C_{2}$ are two symmetrical bicircuits, that is, $\{x^\alpha,y^\beta\}\in E(C_1) \Leftrightarrow \{x^{-\alpha},y^{-\beta}\}\in E(C_2)$), then $\Ft(C_{1})=\Ft(C_{2})= \Ft(G_{\tau})$. 
Hence in figure \ref{FIG7} $G_\tau$ has only one transitive reduction, $C_1$, but $\Ft(G_{\tau})$ has both $C_1$ and $C_2$ as transitive reductions.
}
\end{remark}

\begin{center}
\begin{figure}[h]
\begin{picture}(140,28)(-5,0)
\multiput(0,0)(20,0){2}{\circle*{2}}
\put(10,20){\circle*{2}}
\put(0,0){\line(1,0){20}}
\put(0,0){\line(1,2){10}}
\put(20,0){\line(-1,2){10}}
\put(-1,-5){$2$}
\put(19,-5){$3$}
\put(9,22){$1$}
\multiput(2,-3)(12.5,0){2}{\rotatebox{0}{$+$}}
\multiput(-1.5,2)(6.8,12.5){2}{\rotatebox{60}{$-$}}
\put(11,19){\rotatebox{-60}{$+$}}
\put(18,4.5){\rotatebox{-60}{$-$}}
\curve(0,0, -1,11, 10,20)
\multiput(-4.5,1.5)(9,18){2}{\rotatebox{112}{$+$}}
\put(8,-10){$G_{\tau}$}
\hspace*{3cm}

\multiput(0,0)(20,0){2}{\circle*{2}}
\put(10,20){\circle*{2}}
\put(0,0){\line(1,0){20}}
\put(0,0){\line(1,2){10}}
\put(20,0){\line(-1,2){10}}
\put(-1.5,-5){$2$}
\put(20,-5){$3$}
\put(9,22){$1$}
\multiput(3,1)(11,0){2}{\rotatebox{0}{$+$}}
\multiput(-1.5,2)(6.8,12.5){2}{\rotatebox{60}{$-$}}
\put(11,17.5){\rotatebox{-60}{$+$}}
\put(18,4.5){\rotatebox{-60}{$-$}}
\curve(0,0, -1,11, 10,20)
\curve(10,20, 22,12.5, 20,0)
\curve(0,0, 11,-5, 20,0)
\multiput(-4.5,1.5)(9,18){2}{\rotatebox{112}{$+$}}
\put(0,-3.5){\rotatebox{140}{$-$}}
\put(16.5,-5.25){\rotatebox{50}{$-$}}
\put(12,20){\rotatebox{-30}{$-$}}
\put(21,2){\rotatebox{-20}{$+$}}
\put(4,-10){$\Ft(G_{\tau})$}
\hspace*{3cm}

\multiput(0,0)(20,0){2}{\circle*{2}}
\put(10,20){\circle*{2}}
\put(0,0){\line(1,0){20}}
\put(0,0){\line(1,2){10}}
\put(20,0){\line(-1,2){10}}
\put(-1,-5){$2$}
\put(19,-5){$3$}
\put(9,22){$1$}
\multiput(2,-3)(12.5,0){2}{\rotatebox{0}{$+$}}
\multiput(-1.5,2)(6.8,14){2}{\rotatebox{60}{$-$}}
\put(11,19){\rotatebox{-60}{$+$}}
\put(18,4.5){\rotatebox{-60}{$-$}}
\put(-2,-10){$C_{1}=\Rt(G_\tau)$}
\hspace*{3cm}

\multiput(0,0)(20,0){2}{\circle*{2}}
\put(10,20){\circle*{2}}
\put(0,0){\line(1,0){20}}
\put(0,0){\line(1,2){10}}
\put(20,0){\line(-1,2){10}}
\put(-1,-5){$2$}
\put(19,-5){$3$}
\put(9,22){$1$}
\multiput(2,-3)(12.5,0){2}{\rotatebox{0}{$-$}}
\multiput(-1.5,2)(6.8,14){2}{\rotatebox{60}{$+$}}
\put(11,19){\rotatebox{-60}{$-$}}
\put(18,4.5){\rotatebox{-60}{$+$}}
\put(-5,-10){$C_{2}=\Rt(\Ft(G_\tau))$}
\end{picture}
\vspace{.8cm}
\caption{Example for Remark \ref{rm2}.  $C_1$ is an $\Rt(G_\tau)$ and an $\Rt(\Ft(G_\tau))$.  $C_2$ is an $\Rt(\Ft(G_\tau))$, but not an $\Rt(G_\tau)$ because it is not contained in $G_\tau$.  For legibility, in $\Ft(G_\tau)$ we do not show the positive loops that exist at every vertex.}\label{FIG7}
\end{figure}
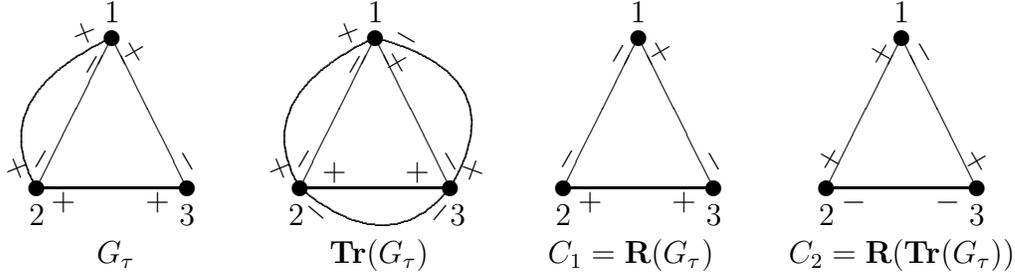
\end{center}

Let $G_{\tau}$ be a bidirected graph and $E = \{e_{1}, e_{2}, \ldots, e_{n}\}$ its set of edges.   Assume $E$ is linearly ordered by a linear ordering $<$ in index order, i.e., $e_i<e_j \Leftrightarrow i<j$.  
Let $(G_{\tau})_{i}$ be a family of graphs constructed from $G_{\tau}$ as follows:
\begin{align*}
(G_{\tau})_{0}&=G_{\tau} \text{ and }\\
(G_{\tau})_{i}&=\begin{cases}
(G_{\tau})_{i-1}-e_{i} &\text{ if } e_{i}= \{x^{\alpha}, y^{\beta}\}\text{ is implied by transitive closure}\\
&\text{ of  a bipath } P_{(\alpha,\beta)}(x, y)\text{ in }(G_{\tau})_{i-1}-\{e_i\},\\
(G_{\tau})_{i-1} &\text{ otherwise.}
\end{cases} 
\end{align*}
We put 
$$
S_{<}(G_{\tau})=\{ e_i \in E : e_{i}\in (G_{\tau})_{i-1} \text{ and } e_i \notin (G_{\tau})_{i}\}.
$$ 

\begin{proposition}\label{24}
Let $G_{\tau}$ be a bidirected graph. 
For each linear ordering $<$ of $E$, $G_{\tau} - S_{<}(G_{\tau})$ is a transitive reduction of $G_\tau$.  
Conversely, if\/ $\Rt(G_{\tau})=(V,\Rt(E); \tau)$ is a transitive reduction of $G_{\tau}$, then $\Rt(G_\tau)=G_\tau-S_{<}(G_{\tau})$ for some linear ordering $<$ of $E$.
\end{proposition}

\begin{proof}
Assume a linear ordering $<$ of $E$.  By construction, if $e_i \in S_{<}(G_\tau)$, then $e_i \in \Ft((G_{\tau})_{i})$.  
Let $m=|E|$; then $\Ft(G_{\tau}) = \Ft^m(G_{\tau} - S_{<}(G_{\tau}))$ (the $m$-times iterate of $\Ft$) $= \Ft(G_{\tau} - S_{<}(G_{\tau}))$ by Proposition \ref{ft2}.  
By Proposition \ref{20a}, $\ft(G_{\tau}; G_{\tau} - S_{<}(G_{\tau})) = G_\tau$ since $G_\tau \subseteq \Ft(G_{\tau}) = \Ft(G_{\tau} - S_{<}(G_{\tau}))$.  
If $G_\tau-S_{<}(G_{\tau})$ were not a minimal partial graph that generates $G_\tau$ under $\Ft$, then there would be an edge $e_j \in E - S_{<}(G_{\tau})$ such that $e_j$ is implied by a bipath $P_{(\alpha,\beta)}(x,y)$ in $G_{\tau} - S_{<}(G_{\tau})-e_j$.  
This bipath is in $(G_\tau)_{j-1}-e_j$ so by construction $e_j \notin (G_\tau)_{j}$, therefore $e_j \in S_{<}(G_{\tau})$, which is absurd.  Therefore $G_\tau-S_{<}(G_{\tau})$ is a transitive reduction of $G_\tau$.

Suppose $\Rt(G_\tau)$ is a transitive reduction of $G_\tau$.  Let $S = E - E(\Rt(G_\tau))$.  Every edge in $S$ is implied by a bipath in $\Rt(G_\tau)$.  Linearly order $E$ by $<$ so that $S=\{e_1,\ldots,e_k\}$ is initial in the ordering.  Then at step $i\leq k$ of the construction of $S_{<}(G_\tau)$, edge $e_i$ is implied by a bipath in $(G_{\tau})_{i-1}-\{e_i\}$ so $e_i \in S_{<}(G_\tau)$; but at step $i>k$, $(G_{\tau})_{i} = \Rt(G_\tau)$, which has no such bipath because of minimality of $\Rt(G_\tau)$.
\end{proof}

\begin{corollary}\label{25}
Let $G_{\tau}$ be a bidirected graph and $<$ a linear ordering of $E(G_\tau)$.  If $S_{<}(G_{\tau})=\emptyset$, then $\Rt(G_{\tau})=G_{\tau}$.
\end{corollary}

\begin{corollary}\label{27}
If $P_{(\alpha,\beta)}(x,y)$ is a bipath, then
$\Rt(P_{(\alpha,\beta)}(x, y))=P_{(\alpha,\beta)}(x, y).$
\end{corollary}

\begin{proof}
For the graph $P_{(\alpha,\beta)}(x, y)$, we have $S_{<}(G_{\tau})=\emptyset$.
\end{proof}

\subsection{Transitive Closure -- Transitive Reduction}

In this section we study the relationship between transitive closure and transitive reduction.

\begin{proposition}\label{31a}
Let $G_{\tau}$ be a bidirected graph.  Then every transitive reduction of $G_\tau$ is a transitive reduction of $\Ft(G_{\tau})$.
\end{proposition}

\begin{proof}
We apply Proposition \ref{24}.  Let $\Rt(G_\tau)$ be a transitive reduction of $G_\tau$.  Choose a linear ordering of $E(\Ft(G_{\tau})$ in which the edges of $\Ft(G_{\tau})-E(G_{\tau})$ are initial and the edges of $\Rt(G_\tau)$ are final.  By the definition of $\Ft$, the $m$ edges of $\Ft(G_{\tau})-E(G_{\tau})$ are in $S_{<}(\Ft(G_{\tau})-E(G_{\tau}))$ and $(\Ft(G_{\tau}))_m = \Ft(G_\tau)$.  The proposition follows.
\end{proof}

It may not be true that every transitive reduction of $\Ft(G_{\tau})$ is an $\Rt(G_\tau)$.  Let $H_\tau$ be a bidirected graph that has more than one transitive reduction, and let $G_\tau = \Rt(H_\tau)$.  Then $G_\tau = \Rt(G_\tau)$.   Since $H_\tau \subseteq \Ft(H_\tau) = \Ft(G_\tau)$, every transitive reduction of $H_\tau$ is a transitive reduction of $\Ft(G_\tau)$, but only one of those transitive reductions can be $G_\tau = \Rt(G_\tau)$.  That cannot happen if $G_\tau$ has no bicircuit.  We prove a lemma first.

\begin{lemma}\label{Ft no bicircuit} 
Let $G_{\tau}$ be a bidirected graph without a bicircuit or parallel edges.  Then $\Ft(G_{\tau})$ has no bicircuit.
\end{lemma}
\begin{proof}
Suppose $e = \{x^\alpha,y^\beta\} \in \Ft(G_\tau)-E(G_\tau)$.  That means there is a bipath $P_{(\alpha,\beta)}(x,y)$ in $G_\tau$.  Now suppose there is a bicircuit $C$ from $z^\gamma$ to $z^{-\gamma}$ in $G_\tau \cup \{e\}$, $C = C_1eC_2$, where we may assume $C_1$ ends at $x^{-\alpha}$ and $C_2$ begins at $y^{-\beta}$.  Then $C_1PC_2$ is a closed bichain from $z^\gamma$ to $z^{-\gamma}$ in $G_\tau$, so it contains a bicircuit, but that is absurd.  Therefore, $G_\tau \cup \{e\}$ contains no bicircuit.  
The proof follows by induction on the number of edges in $\Ft(G_\tau)-E(G_\tau)$.
\end{proof}

\begin{proposition}\label{31}
Let $G_{\tau}$ be a bidirected graph without a bicircuit or parallel edges.  Then $\Rt(\Ft(G_{\tau}))=\Rt(G_{\tau})$.  That is, the unique transitive reduction of $\Ft(G_{\tau})$ is the (unique) transitive reduction of $G_\tau$.
\end{proposition}
\begin{proof}
By Lemma \ref{Ft no bicircuit}, $\Ft(G_{\tau})$ has no bicircuit.  According to Proposition \ref{Rt no bicircuits}, $\Ft(G_{\tau})$ has a unique transitive reduction.  $\Rt(G_\tau)$ is such a transitive reduction.  Therefore, $\Rt(\Ft(G_{\tau}))=\Rt(G_{\tau})$.
\end{proof}

\begin{corollary}\label{32}
Let $G_{\tau}$ be a bidirected graph without a bicircuit or parallel edges. If $S_{<}(G_{\tau})=\emptyset$, then $\Rt(\Ft(G_{\tau}))=G_{\tau}$.
\end{corollary}

\begin{proof}
If $S_{<}(G_{\tau})=\emptyset$, then according to Corollary \ref{25} we have $\Rt(G_{\tau})=G_{\tau}$.
Moreover, it follows from Proposition \ref{31} that $\Rt(\Ft(G_{\tau}))=\Rt(G_{\tau})$, from which the result follows.
\end{proof}

\begin{proposition}\label{33}
Let $G_{\tau}$ be a bidirected graph without a bicircuit or parallel edges.  Then $\Ft(\Rt(\Ft(G_{\tau})))=\Ft(G_{\tau})$.
\end{proposition}

\begin{proof}
By Proposition \ref{31}, 
$\Rt(\Ft(G_{\tau}))=\Rt(G_{\tau}) 
\Rightarrow \Ft(\Rt(\Ft(G_{\tau})))=\Ft(\Rt(G_{\tau}))$
$\Rightarrow \Ft(\Rt(\Ft(G_{\tau})))=\Ft(G_{\tau})$ by Proposition \ref{22}.
\end{proof}

%555555555555555555555555555555555555555555555555555555555555555555555555555555555555555555555555555555555555555555555555555555555555555555555555

\section{The Matroid of a Bidirected Graph}\label{matroid}

We indicate by $b(G_{\tau})$ the number of balanced connected components of $G_{\tau}$.

\begin{theorem}\label{34}{\rm\cite{zs}}
Given a signed graph $G_{\sigma}$, there is a matroid $M(G_{\sigma})$
associated to $G_{\sigma}$, such that a subset $F$ of the edge set
$E$ is a circuit of $M(G_{\sigma} )$ if, and only if, either 
\begin{enumerate}[{\rm Type (i)}]
\item $F$ is a positive cycle,  or
\item $F$ is the union of two negative cycles, having exactly one common vertex,  or
\item $F$ is the union of two vertex-disjoint negative cycles and an elementary chain which is internally disjoint from both cycles.
\end{enumerate}
 The rank function is $r(M(G_{\tau}))=|V|-b(G_{\tau})$.
\end{theorem}

This matroid is now called the \emph{frame matroid} of $G_\sigma$.  See figure \ref{FIG2}, where we represent a positive (resp., negative) cycle by a quadrilateral (resp., triangle).

\setlength{\unitlength}{.7mm}

\begin{figure}[h]
\begin{center}
\begin{picture}(135,30)(0,-5)
\put(15,0){\circle*{2}}
\multiput(0,10)(30,0){2}{\circle*{2}}
\put(15,20){\circle*{2}}
\multiput(15,0)(15,10){2}{\line(-3,2){15}}
\multiput(0,10)(15,-10){2}{\line(3,2){15}}
\put(6,-8){{Type (i)}}
\put(40,0){\line(0,1){20}}
\put(40,0){\line(3,2){30}}
\put(40,20){\line(3,-2){30}}
\put(70,0){\line(0,1){20}}
\multiput(40,0)(15,10){3}{\circle*{2}}
\multiput(40,20)(30,-20){2}{\circle*{2}}
\put(45,-8){{Type (ii)}}
\put(80,0){\line(0,1){20}}
\put(80,0){\line(3,2){15}}
\put(80,20){\line(3,-2){15}}
\put(95,10){\line(1,0){15}}
\put(110,10){\line(3,-2){15}}
\put(110,10){\line(3,2){15}}
\put(125,0){\line(0,1){20}}
\multiput(80,0)(0,20){2}{\circle*{2}}
\multiput(95,10)(15,0){2}{\circle*{2}}
\multiput(125,0)(0,20){2}{\circle*{2}}
\put(91,-8){{Type (iii)}}
\end{picture}
\caption{}\label{FIG2}
\end{center}
\end{figure}
The matroid associated to the bidirected graph is the matroid associated to its signed graph (given by Definition \ref{6}).

\begin{definition}{\rm(\cite{k}; \cite[Def.\ 1.1 of \S3.1]{ka})}\label{35}
{\rm
A signed graph $G_{\sigma}$ is called \emph{quasibalanced} (\emph{m-balanced} in \cite{k,ka}) if it does not admit circuits of types (ii) and (iii).
We have the same definition for bidirected graphs.
}
\end{definition}

\begin{proposition}\label{36}
A connected signed graph $G_{\sigma}$ is quasibalanced if, and only if, for any two negative cycles $C$ and $\acute{C}$ we have 
$|V(C) \cap  V(\acute{C})|\geq 2$.
\end{proposition}
\begin{proof}
Sufficiency results from Definition \ref{35} and Theorem \ref{34}.

To prove necessity, suppose that $G_{\sigma}$  admits two negative cycles $C$ and $\acute{C}$.  
Suppose $|V(C) \cap  V(\acute{C})|= 0$.  Since $G_{\sigma}$ is connected, there exists a chain connecting a vertex of $C$ with a vertex of $\acute{C}$, therefore there exists a circuit of type (iii) which contains  $C$ and $\acute{C}$.
Suppose $|V(C) \cap  V(\acute{C})|= 1$.  Then $C \cup \acute{C}$ is a  circuit of type (ii).  Both cases are impossible; therefore $|V(C) \cap  V(\acute{C})| >  1$.
\end{proof}

\begin{proposition}\label{37}
If $G_{\tau}$ is a quasibalanced bidirected graph, then $\Rt(G_{\tau})$ is quasibalanced.
\end{proposition}

\begin{proof}
As $M(G_{\tau})$ is without circuits of type (ii) and (iii), and since $\Rt(G_{\tau})$ is a partial graph of $G_{\tau}$, the result follows.
\end{proof}

\begin{proposition}\label{38}
Let $G_{\tau}$ be a bidirected graph such that $\Rt(G_{\tau})$ is quasibalanced.  
Let $<$ be a linear ordering of $E$.  
If $G_{\tau}$ is quasibalanced, then for every edge $e$ belonging to the set $S_{<}(G_{\tau})$, $e$ is not in the transitive closure of any cyclic bipath in $G_{\tau}$.
\end{proposition}

\begin{proof}
Let $G_\sigma$ be the signed graph corresponding to $G_\tau$.

Assume that there is an edge $e=\{x^{\alpha},y^{\beta}\}$, belonging to the set $S_{<}(G_{\tau})$, which is in the transitive closure of the cyclic bipath $P = P_{(\alpha,\beta)}(x, y)$, containing the negative cycle $C$.  This implies that $P \cup \{e\}$ is a circuit of the type (ii) or (iii) of $G_{\sigma}$, if the cycle $\acute C$ of $P \cup \{e\}$ that contains $e$ is negative.  By Proposition \ref{bwalksign} the sign of $P$ is $-\alpha\beta$, which is also the sign of $e$.  Therefore, the sign of the closed chain $Pe$ is $+$.  This sign equals $\sigma(C)\sigma(\acute C)$, so $\sigma(\acute C) = \sigma(C) = -$.  
Thus $G_{\tau}$ is not quasibalanced, which is absurd.
\end{proof}

We do not have a sufficient condition for quasibalance.  The converse of Proposition \ref{38} is false.  
Consider $G_\tau$ with $V = \{1,2,3,4,5,6,7\}$ and edges
$$e=1^-2^+,\ 2^+3^+,\ 3^+4^+,\ 4^+5^+,\ 1^-5^+,\ 5^-6^+,\ 6^-7^+,\ 5^+7^+,\ 7^-2^+,$$
linearly ordered in that order.  We claim that $e$ is redundant using the path $P: 15672$, and that no other edge is redundant.   There is no matroid circuit of type (ii) or (iii) in $G_\tau-e$.  But $G_\tau-7^-2^+$ is a matroid circuit of type (ii).  Therefore, $G_\tau$ is not quasibalanced, but $\Rt(G_\tau) = G_\tau - e$ is quasibalanced.
However, $e$ is not in the transitive closure of any cyclic bipath. $P$ and $e$ are the only bipaths from $1$ to $2$.

Let $\overline{F}$ denote the closure of $F$ in a matroid.

\begin{lemma}\label{39}
Let $G_\tau$ be a bidirected graph with edge set $E$.  
If $e \in E - \Rt(E)$, then $e$ belongs to the closure $\overline{\Rt(E)}$ in $M(G_{\tau})$.
If $e \in E(\Ft(G_{\tau})) - E$, then $e$ belongs to the closure $\overline{E}$ in $M(\Ft(G_{\tau}))$.
\end{lemma}
\begin{proof}
Let $P$ be a bipath in $\Rt(G_{\tau})$ which induces $e \in E$. Then $P \cup \{e\}$ is a matroid circuit of type (i), (ii) or (iii). Thus, $e \in \overline{\Rt(E)}$ in $M(G_{\tau})$.

The second statement follows from the first because in $M(\Ft(G_{\tau}))$, $\overline{E}$ is the closure of $\overline{\Rt(E)}$, which equals $\overline{\Rt(E)}$.
\end{proof}

\begin{theorem}\label{42}
Let $G_{\tau}$ be a bidirected graph. Then 
$$r(M(\Ft(G_{\tau})))= r(M(G_{\tau}))= r(M(\Rt(G_{\tau}))).$$
\end{theorem}

\begin{proof}
For $r(M(\Ft(G_{\tau})))= r(M(G_{\tau}))$, it is enough to use Lemma \ref {39}.

For $r(M(\Rt(G_{\tau})))= r(M(\Ft(G_{\tau})))$, it is enough to cite Proposition \ref{22} and replace $G_\tau$ in the previous case by $\Rt(G_\tau)$.
\end{proof}

The definitions of a connected matroid in \cite{o} apply to the matroids of signed graphs.  In particular:

\begin{definition}\label{40}
{\rm
Let $G_{\sigma} = (V, E; \sigma)$ be a signed graph. The matroid $M(G_{\sigma})$ is \emph{connected} if each pair of distinct edges  $e$ and \emph{\'e} from $G_{\sigma}$, is contained in a circuit $C$ of M($G_{\sigma}$).
}
\end{definition}

\begin{theorem}\label{41}
Let $G_{\tau}=(V,E; \tau)$ be a bidirected graph and let $\Rt(G_{\tau})$ be any transitive reduction of $G_{\tau}$. If $M(\Rt(G_{\tau}))$ is connected, then $M(G_{\tau})$ is connected.
\end{theorem}

\begin{proof}
Theorem \ref{42} implies that $\overline{E(\Rt(G_{\tau}))} = \overline{E(G_{\tau})} = E(\Ft(G_{\tau}))$ in $M(\Ft(G_{\tau}))$.  It follows by standard matroid theory, since $M(\Rt(G_{\tau}))$ is connected, that $M(G_{\tau})$ and $M(\Ft(G_{\tau}))$ are connected.
\end{proof}

We note that the converse is false.  For example, let $P_{(\alpha, \beta)}(x, y)$ be a bipath of length not less than 2, whose graph is an elementary chain, and let $e$ be the edge $\{x^{\alpha}, y^{\beta}\}$.  Let $G_\tau = P_{(\alpha, \beta)}(x, y) \cup \{e\}$.  Then $P_{(\alpha, \beta)}(x, y) = \Rt(G_\tau)$, but $M(P_{(\alpha, \beta)}(x, y))$ is disconnected while $M(G_\tau)$ is connected (since the corresponding signed graph is a positive cycle).

\bigskip\bigskip

{\small
{\em Authors' addresses}: \\
{\em Ouahiba Bessouf}, Facult\'e de Math\'ematiques, USTHB BP 32 El Alia, Bab-Ezzouar 16111, Alger, Alg\'erie, 
 e-mail: \texttt{obessouf@\allowbreak yahoo.fr}\\
{\em Abdelkader Khelladi}, Facult\'e de Math\'ematiques, USTHB BP 32 El Alia, Bab-Ezzouar 16111, Alger, Alg\'erie, 
 e-mail: \texttt{kader{\_}khelladi@\allowbreak yahoo.fr} \\
{\em Thomas Zaslavsky}, Department of Mathematical Sciences, Binghamton University, Binghamton, NY 13902-6000, U.S.A., 
 e-mail: \texttt{zaslav@\allowbreak math.binghamton.edu}.
}

\end{document}